\documentclass{amsart}

\usepackage{amsmath}
\usepackage{amsthm}
\usepackage{amsfonts}
\usepackage{hyperref}
\usepackage[all]{hypcap}[2006/02/20]

\newcommand{\mb}[1]{\mathbf{#1}}
\newcommand{\bb}[1]{\mathbb{#1}}

\newcommand{\pard}[2]{\frac{\partial #1}{\partial #2}}
\newcommand{\mbf}[1]{\mathbf{#1}}

\newcommand{\ho}{\left(\frac{d}{dt} -\Delta \right)}
\newcommand{\ddt}[1]{\frac{ d #1}{dt}}
\newcommand{\ip}[2]{\left \langle #1 , #2 \right\rangle}

\newcommand{\n}{\nabla}

\renewcommand{\l}{\lambda}

\renewcommand{\a}{\alpha}
\newcommand{\Si}{\Sigma}
\newcommand{\ra}{\rightarrow}
\newcommand{\ov}{\overline}

\newcommand{\mt}{{\mu^\top}}
\newcommand{\ns}{{\nu^\Si}}

\begin{document}
\theoremstyle{plain}
\newtheorem{theorem}{Theorem}
\newtheorem{lemma}[theorem]{Lemma}
\newtheorem{claim}[theorem]{Claim}
\newtheorem{proposition}[theorem]{Proposition}
\newtheorem{cor}[theorem]{Corollary}
\theoremstyle{definition}
\newtheorem{defses}{Definition}
\newtheorem{assumption}{Assumption}
\theoremstyle{remark}
\newtheorem{remark}{Remark}

\title[A remark on gradient estimates]{A remark on gradient estimates for spacelike mean curvature flow with boundary conditions}
\author{Ben Lambert}
\address{Zukunftskolleg, University of Konstanz, Box 216, 78457 Konstanz - Germany}
\email{benjamin.lambert@uni-konstanz.de}
\maketitle

\begin{abstract}
 We prove a gradient estimate for graphical spacelike mean curvature flow with a general Neumann boundary condition in dimension $n=2$. This then implies that the mean curvature flow exists for all time and converges to a translating solution. 
\end{abstract}

\begin{center}
\end{center}

\section{Introduction}

In this paper we obtain gradient estimates for mean curvature flow (MCF) with general Neumann boundary angle conditions in Minkowski space for dimension $n=2$, leading to existence of the flow for all time and convergence to a translating solution.

In Euclidean space this problem is well studied. In dimension $n=2$, S. Altschuler and L. Wu \cite{AltschulerWu} demonstrated the Euclidean counterpart of this paper demonstrating that graphical MCF with fixed boundary angles exists for all time and converges to a translating solution. Further gradient estimates were also obtained in higher dimensions by B. Guan \cite{BoGuan} demonstrating long time existence, although these depend on the height of the graph and so are not suitable for convergence of the flow, and further estimates for graphs over killing vector fields have been obtained by J. Lira and G. Wanderly \cite{LiraWanderly}. Further results on gradient estimates in Euclidean space have been obtained by G. Huisken \cite{huiskengraph}, A. Stahl \cite{Stahlfirst}, V. Wheeler \cite{WheelerHalfspace}\cite{WheelerRotSym} and the author \cite{LambertTorus}.  

In semi-Riemannian spaces, K. Ecker and G. Huisken \cite{EckerHuiskenCMCMink} demonstrated that MCF (and related flows) may be used to construct prescribed mean curvature hypersurfaces surfaces and in higher codimensions G. Li and I. Salavessa \cite{LiSalavessa} showed that MCF may be applied to find when mappings between Riemannian manifolds are topologically trivial (under some curvature conditions). The Dirichlet boundary value problem for such flows in codimension 1 has been studied by K. Ecker \cite{EckerMinkowskiDBC}\cite{EckerNull}. The perpendicular Neumann boundary condition was considered by the author in several settings \cite{LambertMinkowski}\cite{LambertConstruction}\cite{Lambertgraph}. A recent article by G. Li, B. Gao and C. Wu\cite{GaoLiWu} dealt exactly with the problem of general graphical angle conditions described below for general dimension $n$, however the key boundary gradient lemma in this paper is incorrect. Specifically equation (2.9) in that paper appears to come from differentiating the boundary condition in the normal direction into the domain where no such boundary condition holds. Here the author uses methods similar to Altschuler and Wu's \cite{AltschulerWu} to provide an alternative proof for this result in the restricted case of $n=2$.

Let $\Omega\subset \bb{R}^n$ be a compact convex domain with smooth boundary $\partial \Omega$, where we will take $\bb{R}^n\subset \bb{R}^{n+1}_1$ to be perpendicular to the vector $e_{n+1}$ where $\ip{e_{n+1}}{e_{n+1}}=-1$.  Let $\Sigma$ be the cylinder over $\partial \Omega$ in the direction $\mb{e}_{n+1}$. Define $\gamma$ to be both the outward pointing unit normal to $\partial \Omega$ and  $\mu$ to be the extension of this to the the outward unit normal to $\Sigma$.

Let $M^n$ be an n-dimensional disk with boundary $\partial M^n$ and $\mathbf{F}: M^n \times [0,T) \rightarrow \mathbb{R}^{n+1}_1$ be a smooth map such that $\mbf{F}(\cdot,t)$ is a spacelike embedding of $M^n$ into Minkowski space for all $t \in [0,T)$. Let $\alpha:\Sigma \ra \bb{R}$ be a smooth function which will be used to prescribe the boundary angle. Suppose we are given the spacelike smooth initial embedding $\mathbf{F}_0:M^n \rightarrow \mathbb{R}^{n+1}_1$ then $\mbf{F}$ moves by MCF with $\alpha$-Neumann boundary angle condition if
\begin{equation}
\label{mac}
\begin{cases}
\ip{\frac{d \mathbf{F}}{d t}(p,t)}{\nu(p,t)} = -H (p,t) & \forall (p,t) \in M^n \times [0,T)\\
\mathbf{F}(p,0) = \mathbf{F}_0(p) & \forall p \in M^n\\
\mathbf{F}(p,t) \subset \Sigma & \forall (p,t) \in \partial M \times [0,T)\\
\ip{\nu(p, t)}{ \mu(\mb{F}(p, t)} = \alpha(\mb{F}(p,t)) & \forall (p,t) \in \partial M \times [0,T)
\end{cases}
\end{equation}
where $\nu$ is a smooth normal to the embedding of $M_t=\mb{F}(M^n,t)$. We will assume from now on that $\mb{F}_0$ is smooth and satisfies compatibility conditions, namely that for $p\in\partial M^n$,  $\mb{F}_0(p) \in \Sigma$ and $\ip{\nu(p)}{\mu(F_0(p)}=\alpha$. We remark that the inner product formulation of the first line of (\ref{mac}) is necessary as otherwise (in general) we would require the boundary of $M^n$ to vary with time, as with the usual formulation of MCF, the parametrisation would flow ``out of'' the interior of $\Sigma$.   

We will write $\nabla$ for the connection on $M_t$, $\nabla^\Si$ for the connection on $\Sigma$, and the ambient connection on $\bb{R}^{n+1}_1$ will be denoted $\ov \n$. The second fundamental form on $M$ and $\Sigma$ will be written $A(X,Y) =: \ip{\ov \n_X \nu}{Y}$ and $A^\Si(X,Y) = \ip{\ov\n_X \mu}{Y}$ respectively. We observe that when $n=2$, since $\Sigma$ is a cylinder, $A^\Si(X,Y) = \ip{X + \ip{X}{e_{n+1}}e_{n+1}}{Y + \ip{Y}{e_{n+1}}e_{n+1}}\kappa$ where $\kappa$ is the curvature of the curve defined by $\partial \Omega\subset \bb{R}^2$. We will say that $\partial \Omega\subset \bb{R}^2$ is strictly convex if $\kappa>0$.  

We will say $\alpha$ is a \emph{graphical} boundary angle if for all $p\in\Sigma$, $ \n_{e_{n+1}}^\Si \alpha|_p=0$, that is the boundary angle does not vary in the $e_{n+1}$ direction. If $\mb{F}_0$ is spacelike then we may represent $\mbf{F}_0$ as a graph $u_0: \Omega \rightarrow \mathbb{R}$ initially with the derivative bound $|Du_0|<1$. If in addition $\alpha$ is graphical, equation (\ref{mac}) is equivalent (by an argument identical to \cite[Section 1]{EckerHuiskenInteriorEstimates}) to finding $u:\Omega \times [0,T)\ra \bb{R}$ such that
\begin{equation}
\label{parmac}
\begin{cases}
u_t= \sqrt{1-|Du|^2} D_i \left( \frac{ D_i u }{\sqrt{1- |Du|^2}} \right) & \forall ({x},t) \in \Omega \times [0,T)\\
u({x}, 0)=u_0({x}) & \forall {x} \in \Omega\\
\gamma^i D_iu ({x},t) = \sqrt{1-|Du|^2}\alpha(x,t) & \forall ({x},t) \in \partial \Omega \times [0,T)\ \ .
\end{cases}
\end{equation}

We define a translating solution to (\ref{mac}) to be one which stays the same up to reparametrisation and translation depending on time. This may be viewed as a solution of (\ref{parmac}) of the form $\tilde u(x,t) = \tilde{u}(x,0) +\l t$ for some $\l$. 

A key ingredient to demonstrating uniform parabolicity to equation (\ref{parmac}) is finding a gradient estimate, such that there is a constant $C$ depending only on the initial data, $\alpha$ and $\Omega$ such that $|Du|(x,t)\leq C<1$ for all the time the flow exists. Equivalently we require an upper estimate on
\[v:= -\ip{\nu}{e_{n+1}} = \frac{1}{\sqrt{1-|Du|^2}}\ \ .\]
We obtain this estimate in Proposition \ref{Gradest}, and as a corollary we obtain the following:

\begin{theorem}
 Suppose $\Omega$ is a smooth strictly convex domain, and the boundary angle prescription function $\alpha$ is graphical. Then any solution to equation (\ref{mac}) starting from smooth spacelike initial data exists for all time and converges uniformly to a translating solution as $t\ra \infty$.  
\end{theorem}
\begin{proof}
The boundary condition in equation (\ref{parmac}) is oblique, and as a result of Proposition \ref{Gradest}, the flow is uniformly parabolic with a uniform gradient estimate. Therefore, we may apply methods such as in \cite[Section VIII.3]{Lieberman} to obtain existence for as long as $|u|$ is bounded. We may get an explicit estimate on $|u|$ by observing that the first line in equation (\ref{parmac}) gives $u_t = Hv^{-1}$, and so due to Lemma \ref{Hest} we have $|u_t|<C_H$, immediately implying that at time $t$
\begin{equation}|u(x,t)| = C(u_0)+tC_H\ \ .
 \label{uest}
\end{equation}
We conclude that a solution to (\ref{parmac}) exists for all time.

Since we have a gradient estimate that is unform in time and a height bound of the form (\ref{uest}), both existence of a translating solution $\tilde{u}$ to (\ref{parmac}) and the convergence to $\tilde{u}$ may now be seen by following a strong maximum principle argument as in \cite[Section 6.2]{Schnuerersecond}. Here, we do not rewrite proof, as the arguments in \cite{Schnuerersecond} carry over with only trivial modifications. More precisely the only difference is that we obtain the initial linear equation and boundary condition for $w$ on bottom of p340 and top of p341 of \cite{Schnuerersecond} from (\ref{parmac}), which is quasilinear with a uniformly oblique boundary condition, meaning that an identical equation follows easily by standard methods. Otherwise the proof of existence of a translating solution and convergence to that solution is identical.
\end{proof}

\section{The boundary condition}

In this section we consider the effect of the condition
\[\ip{\nu}{\mu} = \alpha\]
where $\a\in C^\infty(\Sigma)$ and $|\n^\Si\a|<C_\Si$ and $\Si$ is strictly convex.
\begin{lemma}\label{spacederiv}
 For $p\in \Si$ and $W\in T \Si \cap T M_t$,
 \[\n^\Si_W \alpha = A(W,\mu^\top)+A^\Si(W,\nu^\Si)\]
\end{lemma}
\begin{proof}
 We calculate (see also \cite[Proposition 2.2]{Stahlsecond}\cite[Lemma 5.2]{LambertMinkowski})
 \[\n^\Si_W \a = W(\ip{\nu}{\mu}) = A(W,\mu^\top)+A^\Si(W,\nu^\Si)\]
\end{proof}
For $p\in \Sigma$ and $X\in T_p\bb{R}_1^{n+1}$, we define projections into $T_p M$, $T_p \Sigma$ and $T_p M \cap T_p \Sigma$ by
\[X^\top = X + \ip{X}{\nu}\nu, \ \  X^\Sigma = X - \ip{X}{\mu}\mu, \ \  X^\tau = X-\frac{\ip{X}{\mu}}{1+\alpha^2}\mu^\top+\frac{\ip{X}{\nu}}{1+\alpha^2}\nu^\Sigma\]
In particular we have 
 \begin{equation}e_{n+1} - v \nu = e_{n+1}^\top = e_{n+1}^\tau -\frac{v\alpha}{1+\alpha^2}\mt\label{linalg1}
 \end{equation}
 and
 \begin{equation}e_{n+1}^\tau =e_{n+1}+v\left( -\nu +\frac{\alpha}{1+\alpha^2}\mt\right) = e_{n+1}-\frac{v}{1+\alpha^2}\nu^\Si  \ \ .\label{linalg2}
 \end{equation}
We recall that $v:=-\ip{\nu}{e_{n+1}}$ where we choose the sign on $e_{n+1}$ so that $v>0$, and observe the following lemma:
\begin{lemma}\label{bdryv}
 At any point in $\Si\cap M_t$, we have
 \[  \n_\mt v =\frac{v}{1+\a^2}\left[\a A(\mt,\mt) -A^\Si(\nu^\Si, \nu^\Si)\right]-\n^\Si_{e_{n+1}^\tau} \alpha\ \ .\]
\end{lemma}
\begin{proof}
 We have that at the boundary
 \[\n_{\mu^\top}v=-A(e_{n+1}^\top,\mt) \ \ .\]
 Using Lemma \ref{spacederiv} and equations (\ref{linalg1}) and (\ref{linalg2}),
 \begin{flalign*}
  \n_\mt v &=A^\Si(e_{n+1}^\tau, \nu^\Si)-\n^\Si_{e_{n+1}^\tau} \alpha+\frac{\a v}{1+\a^2}A(\mt,\mt)\\
  &=-\frac{v}{1+\a^2}A^\Si(\nu^\Si, \nu^\Si)+\frac{\a v}{1+\a^2}A(\mt,\mt)-\n^\Si_{e_{n+1}^\tau} \alpha
 \end{flalign*}
where we also used that $e_{n+1}$ is a zero eigenvector of $A^\Si(\cdot,\cdot)$. 

\end{proof}

\begin{lemma} \label{bdryH}
 At any point in $\Si\cap M_t$, we have
 \[\n_\mt H = \frac{H}{1+\a^2} \left[\a A(\mt,\mt) -A^\Si(\nu^\Si,\nu^\Si) +\n^\Si_{\nu^\Si}\a\right]\]
\end{lemma}
\begin{proof}
 We consider $x(t)\in M^n$ such that $F(x(t),t)$ is constrained to lie on the line $p+s e_{n+1}\subset \Sigma$ for some $s\in \bb{R}$ and $p\in \partial \Omega$. We see that
 \[\ddt{F(x(t),t)} = H\nu +\frac{H}{v}e_{n+1}^\top\]
 because $\ddt{F(x(t),t)}=\l e_{n+1}$ where $\l \ip{e_{n+1}}{\nu} = -H$. We may now see
 \[\ddt{\nu(x(t),t)} = \n H+\frac{H}{v}\n_{e_{n+1}^\top} \nu\]
 where we used that under the flow, $\pard{\nu}{t}=\n H$ (see \cite[Proposition 3.1]{EckerHuiskenCMCMink}). 
 
 We now see that since $e_{n+1}$ is a zero eigenvector of $A^\Si(\cdot, \cdot)$, 
 \begin{flalign*}
\ddt{}&\ip{\nu(x(t))}{\mu(F(x(t),t)} = \n_\mt H +\frac{H}{v}A(e_{n+1}^\top, \mt)\\
 &= \n_\mt H -\frac{H}{v}A^\Si(e_{n+1}^\tau, \nu^\Si)-\frac{\a H}{1+\a^2}A(\mt, \mt)+\frac{H}{v}\n^\Si_{e_{n+1}^\tau}\a\\
  &= \n_\mt H +\frac{H}{1+\a^2}A^\Si(\nu^\Si, \nu^\Si)-\frac{\a H}{1+\a^2}A(\mt, \mt)+\frac{H}{v}\n^\Si_{e_{n+1}^\tau}\a
%  &= \n_\mt H -\frac{H}{1+\a^2}A^\Si(e_{n+1}^\tau, \nu^\Si)-\frac{H\a}{1+\a^2}A(\mt, \mt)
 \end{flalign*}
where we used (\ref{linalg1}) and Lemma \ref{spacederiv} to get the second line and (\ref{linalg2}) to obtain the third. 

Since $\ddt{}\a(F(x(t),t)) = \frac{H}{v}\n^\Si_{e_{n+1}} \a$, the Lemma follows from (\ref{linalg2}).
\end{proof}

\section{Gradient estimate for $n=2$}

We include following for completeness (compare alternative graphical notation proof in \cite[Lemma 2.1]{GaoLiWu}): 

\begin{lemma}\label{Hest}
  If the boundary angle $\alpha$ is graphical, that is for all $p\in \Sigma$, $\n^\Si_{e_{n+1}} \a|_p=0$, then for all time such that the flow exists, 
  \[H^2 \leq C_H^2 v^2\]
  where $C_H=\underset{M_0}\sup \frac{|H|}{v}$. 
\end{lemma}
\begin{proof}
 We have the following well known evolution equations (see e.g. \cite[Proposition 2.3, Proposition 2.6]{EckerMinkowskiDBC}\cite[Proposition 3.2, Proposition 3.3]{EckerHuiskenCMCMink})
 \begin{equation}\ho H = -H|A|^2, \ \ \ho v = -v |A|^2\ \ ,\label{Evols} 
 \end{equation}

 and so on the interior of $M_t$,
 \begin{flalign*}
  \ho \frac{H^2}{v^2}& = -\frac{2|\n H|^2}{v^2} -\frac{6 H^2|\n v|^2}{v^4} +8\frac{H}{v^3}\ip{\n v}{\n H}\\
  & =-\ip{\frac{\n H}{H}}{\n \frac{H^2}{v^2}} + 3\ip{\frac{\n v}{v}}{\n \frac{H^2}{v^2}}
 \end{flalign*}
while meanwhile using (\ref{linalg2}), Lemma \ref{bdryv} and Lemma \ref{bdryH}, we see that at the boundary
\begin{flalign*}
 \n_{\mu^\top} \frac{H^2}{v^2} &= 2\frac{H^2}{v^2}\left[\frac{\n^\Si_\ns \a}{1+\a^2} +\frac{\n^\Si_{e_{n+1}^\tau}\a}{v} \right]\\
 &= 2\frac{H^2}{v^3}\n^\Si_{e_{n+1}} \a\\
 &=0 \ \ .
\end{flalign*}
 
Applying the maximum principle gives the result.
\end{proof}

Similarly to in \cite{AltschulerWu}, the restriction to $n=2$ is now used to estimate the difficult $A(\mu^\top, \mu^\top)$ term in Lemma \ref{bdryv} by $H$ and $\n_{e_{n+1}^\tau} v$. This leads to the following gradient estimate:

\begin{proposition}[Gradient estimate in dimension 2]\label{Gradest}
 Suppose that $n=2$, $\partial \Omega$ is strictly convex and the boundary angle is graphical. Then there exists a time independent constant $C$ depending only on $M_0$, $\partial \Omega$, $\alpha$ and $\nabla^\Sigma \alpha$ such that for all time that the flow exists, 
 \[v \leq C\ \ .\]
\end{proposition}
\begin{proof}
 We aim to apply the maximum principle to $v$, and in view of equation (\ref{Evols}), all we need to show is that at a large boundary maximum, $\n_{\mu^\top} v \leq 0$. We begin by estimating $A(\mu^\top, \mu^\top)$.
 
 Let $p \in M_t \cap \Sigma$ be a boundary maximum of $v$ such that $v(p)\geq C$ where $C$ is to be chosen later and consider $e_{n+1}^\tau$ at this point. Since $|e_{n+1}^\tau|^2 = \frac{v^2}{1+\a^2}-1$,  we see that by choosing $C>2\underset{\partial \Omega} \sup \sqrt{1+\a^2}$ we may assume $e_{n+1}^\tau \neq 0$. 
  
 We calculate that at $p$,
 \begin{flalign}
  0&= \n_{e_{n+1}^\tau} v \nonumber\\&= A(e^\top, e_{n+1}^\tau)\nonumber\\&= A(e_{n+1}^\tau, e_{n+1}^\tau) -\frac{v\a}{1+\a^2}A(\mu^\top, e_{n+1}^\tau)\nonumber\\
  & = A(e_{n+1}^\tau, e_{n+1}^\tau) +\frac{v\a}{1+\a^2}A^\Si(\nu^\Si, e_{n+1}^\tau)-\frac{v\a}{1+\a^2}\n_{ e_{n+1}^\tau}^\Si\a\nonumber\\
  & = A(e_{n+1}^\tau, e_{n+1}^\tau) -\a A^\Si(e_{n+1}^\tau, e_{n+1}^\tau)-\frac{v\a}{1+\a^2}\n_{ e_{n+1}^\tau}^\Si\a \label{tanderiv}
 \end{flalign}
where we used equation (\ref{linalg1}) on the second line, Lemma \ref{spacederiv} on the third and (\ref{linalg2}) on the fourth. 

 Using (\ref{linalg2}) and the fact that $e_{n+1}$ is a zero eigenvector of $A^\Si(\cdot,\cdot)$, we see that
 \begin{equation}A^\Si(\nu^\Si, \nu^\Si) = A^\Si(\nu -\a \mu- ve_{n+1},\nu -\a \mu- ve_{n+1}) = (v^2 -1 -\a^2)\kappa
\label{ASigma}
 \end{equation}

 and
 \[A^\Si(e_{n+1}^\tau,e_{n+1}^\tau) = \frac{v^2}{(1+\a^2)^2}A^\Si(\nu^\Si,\nu^\Si) =  \frac{v^2(v^2-1-\a^2)}{(1+\a^2)^2}\kappa\ \ .\] 
We define $T = |e_{n+1}^\tau|^{-1}e_{n+1}^\tau = \sqrt{1+\a^2}(v^2-\a^2-1)^{-\frac{1}{2}}e_{n+1}^\tau$ and see
 \begin{equation}A^\Si(T,T) = \frac{v^2}{1+\a^2}\kappa\ \ .
  \label{ASigmaT}
 \end{equation}
Similarly if we define $\underset{V\in T_p\partial \Omega, |V|=1}\sup \ov \n_V \alpha = C_\a$ then we observe that
 \begin{equation}|\n^\Si_{e_{n+1}^\tau} \a| \leq v\sqrt{\frac{v^2-1-\a^2}{1+\a^2}}C_\a \ \ .
  \label{grada}
 \end{equation}

Applying (\ref{ASigmaT}) and (\ref{grada}) to (\ref{tanderiv}) gives
\begin{flalign*} A(T,T) &= \frac{\a v^2}{1+\a^2}\kappa +\frac{\a v}{v^2-1-\a^2}\n^\Si_{e_{n+1}^\tau} \a\\& \geq \frac{\a v^2}{1+\a^2}\kappa -\frac{\a v^2}{\sqrt{(1+\a^2)(v^2-1-\a^2)}}C_\a\ \ .
\end{flalign*}

We may therefore use Lemma \ref{Hest} to estimate
\begin{flalign*}
\frac{1}{1+\a^2}A(\mu^\top,\mu^\top)& = H - A(T,T)\\&\leq C_H v -\kappa \frac{\a v^2}{1+\a^2}+\frac{\a v^2}{\sqrt{(1+\a^2)(v^2-1-\a^2)}}C_\a\ \ ,
\end{flalign*}
which we may now apply along with (\ref{ASigma}) and (\ref{grada}) to estimate the right hand side of the boundary derivative of $v$ in Lemma \ref{bdryv}
\begin{flalign*}
\n_\mt v &\leq v\left[\a C_H v +\frac{\a^2 v^2}{\sqrt{(1+\a^2)(v^2-1-\a^2)}}C_\a -\frac{v^2-1}{1+\a^2}\kappa \right]-\n^\Si_{e_{n+1}^\tau} \alpha\\
&\leq v\left[\a C_H v +\frac{\a^2 v^2}{\sqrt{(1+\a^2)(v^2-1-\a^2)}}C_\a +\sqrt{\frac{v^2-1-\a^2}{1+\a^2}}C_\a-\frac{v^2-1}{1+\a^2}\kappa \right]
\end{flalign*}
which is clearly negative for large enough $v$. In particular, while $v>2\underset{\partial \Omega} \sup \sqrt{1+\a^2}$ we may estimate
\[\n_\mt v \leq \frac{v^2}{1+\a^2}\left[ (1+\a^2)(\a C_H +(2\a^2+1)C_\a+\kappa) - v\kappa \right]\]
and so for $v\geq \kappa^{-1}(1+\a^2)(\a C_H +(2\a^2+1)C_\a+\kappa)$ we see that $\n_\mt v\leq 0$. The Lemma follows from the maximum principle by choosing
\[C =\max\{2 \sqrt{1+\overline{\a}^2}, \ (1+\overline{\a}^2)(\underline{\kappa}^{-1}\overline{\a} C_H +\underline{\kappa}^{-1}(2\overline{\a}^2+1)C_\a+1), \ \underset{M_0}{\sup} \,v \}\ \ ,\]
where $\overline \a = \underset{x \in \partial \Omega} \sup |\alpha|$ and $\underline{\kappa} = \underset{x\in \partial \Omega} \inf \kappa >0$.
\end{proof}

\section{Remarks on $n\geq 3$}

It would be interesting to obtain similar estimates in higher codimension, and we observe that Lemma \ref{Hest} did not require a dimensional restriction. Clearly the proof in the previous section will no longer hold, and so we must find some other way of estimating $A(\mu^\top, \mu^\top)$ in Lemma \ref{bdryv}. 

One possible solution observed by B. Guan \cite{BoGuan} in the Euclidean case, is to use an extension of the boundary condition itself to obtain a bound. We extend $\mu$ smoothly to all of $\bb{R}^{n+1}_1$ so that for all $p\in \bb{R}_1^{n+1}$, $\ov \n_{e_{n+1}} \mu |_p=0$  and for all $q\in\Sigma$, $\ov \n_\mu \mu|_q =0$, and define $\tilde \alpha = \ip{\mu}{\nu}$. We may then observe that at the boundary
\[\n_\mt \tilde{\alpha} = \alpha A^\Si(\nu^\Si,\nu^\Si) + A(\mu^\top, \mu^\top) \ \ ,\]
while we also know that at the boundary $\a=\tilde{\a}$. We may aim to estimate functions such as $v(1+\tilde{\a}^2)^{-\frac{1}{2}}$, which have a negative boundary derivative, as required. However, due to the indefinite metric on the ambient space (as opposed to in definite spaces), the group of isometries fixing a point are noncompact. This implies we must must estimate projections with an extra $v$ term, and several such projections appear in the evolution of $\tilde{\a}$. The evolution of $\tilde{\alpha}$ reads
\begin{flalign*}
 \ho \tilde{\a}&=-\tilde{\a}|A|^2-2h^{ij}\ip{\ov \n_\pard{F}{x^i} \mu}{\pard{F}{x^j}}-g^{ij}\ip{\ov\n^2_{\pard{F}{x^i}\pard{F}{x^j}}\mu}{\nu}
 \end{flalign*}
where, in general, no signs may be obtained on the last two terms. These must therefore be estimated by $C_1|A|v^2$ and $C_2v^3$ respectively, and these large powers of $v$ make estimates in general a challenge, and more than can be dealt with purely from the evolution of $v$.

\bibliographystyle{plain}
%\bibliography{$HOME/Documents/bibblee/bib}

\begin{thebibliography}{20}

\bibitem{AltschulerWu}
S. J. ~Altschler and L. F. ~Wu.
\newblock Translating surfaces of the non-parametric mean curvature flow with
  prescribed contact angle.
\newblock {\em Calculus of Variations and Partial Differential Equations},
  2:101--111, 1994.

\bibitem{EckerMinkowskiDBC}
K. ~Ecker.
\newblock Interior estimates and longtime solutions for mean curvature flow of
  noncompact spacelike hypersurfaces in minkowski space.
\newblock {\em Journal of Differential Geometry}, 45:481--498, 1997.

\bibitem{EckerNull}
K. ~Ecker.
\newblock Mean curvature flow of spacelike hypersurfaces near null initial
  data.
\newblock {\em Communications in Analysis and Geometry}, 11:181--205, 2003.

\bibitem{EckerHuiskenInteriorEstimates}
K. ~Ecker and G. ~Huisken.
\newblock Interior estimates for hypersurfaces moving by mean curvature.
\newblock {\em Inventiones mathematicae}, 105:547--569, 1991.

\bibitem{EckerHuiskenCMCMink}
K. ~Ecker and G. ~Huisken.
\newblock Parabolic methods for the construction of spacelike slices of
  prescribed mean curvature in cosmological spacetimes.
\newblock {\em Communications in Mathematical Physics}, 135:595--613, 1991.

\bibitem{BoGuan}
B.~Guan.
\newblock Mean curvature motion of non-parametric hypersurfaces with contact
  angle condition.
\newblock In {\em Elliptic and Parabolic Methods in Geometry}, pages 47--56,
  Wellesley (MA), 1996. A K Peters.

\bibitem{huiskengraph}
Gerhard Huisken.
\newblock Non-parametric mean curvature evolution with boundary conditions.
\newblock {\em Journal of Differential Equations}, 77:369--378, 1989.

\bibitem{Lambertgraph}
B.~Lambert.
\newblock A note on the oblique derivative problem for graphical mean curvature
  flow in {M}inkowski space.
\newblock {\em Abhandlungen aus dem Mathematischen Seminar der Universit{\"a}t
  Hamburg}, 82(1):115--120, 2012.

\bibitem{LambertConstruction}
B.~Lambert.
\newblock Construction of maximal hypersurfaces with boundary conditions.
\newblock ArXiv preprint, submitted to Manuscripta Mathematica, 2014.
\newblock http://arxiv.org/abs/1408.5309.

\bibitem{LambertTorus}
B. ~Lambert.
\newblock The constant angle problem for mean curvature flow inside rotational
  tori.
\newblock {\em Mathematical Research Letters}, 21(3):537 -- 551, 2014.

\bibitem{LambertMinkowski}
B. ~Lambert.
\newblock The perpendicular {N}eumann problem for mean curvature flow with a
  timelike cone boundary condition.
\newblock {\em Transactions of the American Mathematical Society},
  366:3373--3388, 2014.

\bibitem{LiSalavessa}
G.~Li and I.M.C. ~Salavessa.
\newblock Mean curvature flow of spacelike graphs.
\newblock {\em Mathematische Zeitschrift}, 269(3-4):697--719, 2011.

\bibitem{Lieberman}
G.M. ~Lieberman.
\newblock {\em Second Order Parabolic Differential Equations}.
\newblock World Scientific Publishing Co. Pte. Ltd., 1996.

\bibitem{LiraWanderly}
J.H. ~Lira and G.A. ~Wanderly.
\newblock Mean curvature flow of {K}illing graphs.
\newblock {\em Transactions of the American Mathematical Society},
  367:4703--4726, 2015.

\bibitem{GaoLiWu}
G.~Li S.~Gao and C.~Wu.
\newblock Translating spacelike graphs by mean curvature flow with prescribed
  angle.
\newblock {\em Archive der Mathematik}, 103:499--508, 2014.

\bibitem{Schnuerersecond}
O.~Schn{\"u}rer.
\newblock Translating solutions to the second boundary value problem for
  curvature flows.
\newblock {\em Manuscripta Mathematica}, 108:319--347, 2002.

\bibitem{Stahlsecond}
A. ~Stahl.
\newblock Convergence of solutions to the mean curvature flow with a {N}eumann
  boundary condition.
\newblock {\em Calculus of Variations and Partial Differential Equations},
  4:421--441, 1996.

\bibitem{Stahlfirst}
A. ~Stahl.
\newblock Regularity estimates for solutions to the mean curvature flow with a
  {N}eumann boundary condition.
\newblock {\em Calculus of Variations and Partial Differential Equations},
  4:385--407, 1996.

\bibitem{WheelerHalfspace}
V.M. ~Wheeler.
\newblock Mean curvature flow of entire graphs in a half-space with a free
  boundary.
\newblock {\em Journal f{\"u}r die reine und angewandte Mathematik},
  690:115--131, 2014.

\bibitem{WheelerRotSym}
V.M. ~Wheeler.
\newblock Non-parametric radially symmetric mean curvature flow with free
  boundary.
\newblock {\em Mathematische Zeitschrift}, 276(1):281--298, 2014.

\end{thebibliography}

\end{document}